\documentclass[10pt,oneside]{amsart}
\usepackage{amsmath, amsthm, amssymb, graphicx, enumitem}

\newtheorem{thm}{Theorem}
\newtheorem{lem}[thm]{Lemma}

\newtheorem{case}{Case}
\newtheorem{subcase}{Subcase}
\makeatletter\@addtoreset{case}{thm}\makeatother
\makeatletter\@addtoreset{subcase}{case}\makeatother

\theoremstyle{definition}
\newtheorem{defn}{Definition}

\theoremstyle{remark}
\newtheorem*{rem}{Remark}

\begin{document}

\title{Knot theory of Morse-Bott critical loci}
\author{Metin Ozsarfati}
\address{Department of Mathematics, Michigan State University, East Lansing, MI 48824 }
\email{ozsarfat@msu.edu}

\begin{abstract} 
We give an alternative proof of that a critical knot of a Morse-Bott function $f: S^3 \rightarrow \mathbb{R}$ is a graph knot where the critical set of $f$ is a link in $S^3$ \cite{gordon}\cite{morgan}\cite{eisen-neu}. Our proof inducts on the number of index-1 critical knots of $f$. 
\end{abstract}

\maketitle

\section{$k$-functions on $S^3$}

We define  below  a specific Morse-Bott function \textbf{\cite{bott}} $f: S^3 \rightarrow \mathbb{R}$ to study knots in $S^3$ where each critical component of $f$ is a 1-dimensional manifold (circle). Throughout this paper, $f$ or $f_i$ will denote such functions.

\begin{defn}

A real valued smooth function $f: S^3 \rightarrow \mathbb{R}$ is \emph{a $k$-function} if:

\begin{enumerate}[label=(\roman*)]

\item The set of critical points of $f$ is a link $L$ in $S^3$ 

\item The Hessian of $f$ is nondegenerate in the normal direction to $L$. 

\end{enumerate}

\end{defn}

The link $L$ is called \emph{the critical link of $f$} and a component of $L$ is called \emph{a critical knot of $f$}. For each critical knot $K$ of $f$, there exist a tubular neighborhood $U$ of with local coordinates $(\theta,x,y)$ such that $f(\theta,x,y)=  c^2 (\pm x^2 \pm y^2) + d$ where $(\theta,0,0)$ are the coordinates for $K$ and $c, d$ are scalars. We will call such a neighborhood $U$ of $K$ \emph{a $k$-model neighborhood} and make an abuse of notation by identifying $U \subseteq S^3$ with $S^1 \times D^2$ where $D^2$ is the unit disk. The notation $(\theta,x,y)$ will refer to such local coordinates of a critical knot throughout this paper. 

\begin{defn} A link $L$ is \emph{$k$-mate} if $L$ is a subset of the critical link of some $k$-function. \end{defn}

The basic question is then which knots in $S^3$ are $k$-mate. Our main theorem answers this question. 

\begin{thm} \label{k-mate} A knot is $k$-mate if and only if it is a graph knot. 
\end{thm}

A knot $K$ is \emph{a graph knot} if  \emph{the JSJ-decomposition} \textbf{\cite{jsj1}}, \textbf{\cite{jsj2}} of the complement of $K$ are all Seifert fibered pieces. Such $3$-manifolds are called \emph{graph manifolds} \textbf{\cite{wald}}. Another characterization of a graph knot $K$  is that \emph{the Gromov volume} of the complement of $K$ is zero \textbf{\cite{gromov}}, \textbf{\cite{thurston}}. It is shown in \textbf{\cite{gordon}} (Corollary 4.2) that a graph knot $K$ is obtained from an unknot by a finite application of connected sum or cabling operations to it (see Definition \ref{graph knot}). 

Note that graph knots form a small class of knots; e.g. hyperbolic knots do not belong to this class. A proof of Theorem \ref{k-mate} first appears in \textbf{\cite{morgan}} where connected sum operation has been overlooked though. \emph{Graph links} (similarly defined) have been effectively classified in \textbf{\cite{eisen-neu}} (where they are called \emph{Seifert links}). The results in \cite{eisen-neu} are stronger than ours as they study links in a 3-dimensional homology sphere.

\section{Basic notions and results}

The standard notions about $k$-functions we introduce here are actually the usual standard notions about Morse functions on a manifold \textbf{\cite{milnor2}}.  A critical knot $K$ of a $k$-function $f$ is called a source, sink or saddle respectively if the signs in $f(\theta,x,y)= c^2(\pm x^2 \pm y^2) + d$ are both positive, both negative or opposite respectively. We adopt \emph{the sign conventions in} $f(\theta,x,y) =  c^2(y^2 - x^2) + d$ for a saddle. 

In the saddle $K$ case, the two circles $S^1 \times  (\pm 1, 0)$ are called \emph{ stable circles of $K$ in $U$} and the two circles $S^1 \times  (0, \pm 1)$ are called \emph{the unstable circles of $K$ in $U$}. Similarly, the annulus $S^1 \times [-1,1] \times \lbrace 0 \rbrace$ is called \emph{the stable annulus of $K$ in $U$} and the annulus $S^1 \times \lbrace 0 \rbrace \times [-1,1]$ is called \emph{the unstable annulus of $K$ in $U$}. Note that neither $k$-model coordinates $(\theta,x,y)$ nor stable or unstable circles of a saddle are unique. The stable and unstable circles of a saddle $K$ can be isotoped to $K$ within the stable or unstable annulus so that they are parallel cable knots of $K$. They homologically have $\pm 1$ longitude coefficients (and some arbitrary meridian coefficient)  in $H_1(\partial U)$.  

A point in $f(L) \subseteq \mathbb{R}$ where $L$ is the critical link of a $k$-function $f$ is called \emph{a critical value of $f$} and a point in $\mathbb{R} - f(L)$ is called \emph{a regular value of $f$}. We will define \emph{an ordered $k$-function} later and Lemma \ref{torus} shows that the preimage of a regular value of an ordered $k$-function is a collection of disjoint tori in $S^3$.

Given a $k$-function $f$, there exists \emph{a gradient like vector field $X$ on $S^3$ for $f$}. More precisely, $X_p(f)$ is positive if $p$ is not a critical point of $f$ and also, $X(\theta,x,y) = c^2 \cdot (\pm 2x \frac{\partial}{\partial x} \pm 2y \frac{\partial}{\partial y})$ around a critical knot of $f$ (see e.g. \textbf{\cite{milnor1}} for the existence of a gradient like vector field for a Morse function). The function $f$ is increasing on the forward flow lines of $X$. For any point $p$ in $S^3$ , the flow line $X_t(p)$ converges to a critical point of $f$ as $t \rightarrow \pm \infty$. An important application of $X$ is that the flow of $X$ gives an isotopy between the regions $f^{-1}((-\infty, r])$ and $f^{-1}((-\infty, r+\epsilon])$ in $S^3$ (here, $\epsilon >0$) when $f^{-1}([r,r+\epsilon])$ does not contain any critical points of $f$.

A source of a $k$-function is a sink of $-f$ and vice versa. We may always assume a $k$-mate knot to be a source of some $k$-function by the following lemma. 

\begin{lem} \label{source} A knot $K$ is a source of some $k$-function if and only if it is a saddle of some $k$-function.

\end{lem}

\begin{proof} We will prove only one direction as the other one is similar. Suppose that $K$ is a saddle of a $k$-function $f$. Let $\tilde{U}$ be a $k$-model neighborhood of  $K$. Let $D$ be the disk of radius $1/2$ centered at the origin in $\mathbb{R}^2$ and consider the smaller $k$-model neighborhood $U = S^1 \times D$ of $K$ inside $\tilde{U}$.

Consider the isotopic knots $K_1 := S^1 \times  (2/3,0)$ and $K_2 := S^1 \times (5/6,0)$ and take a small tubular neighborhood $V_i$ of $K_i$ in $\text{Int}(\tilde{U})$ so that $V_i$ intersects each meridian disk $\lbrace \theta \rbrace \times D^2$ of $\tilde{U}$ in a disk (See Fig. \ref{source figure}). Moreover, the intersections $V_1 \cap U = \partial V_1 \cap \partial U$ and $V_1 \cap V_2 = \partial V_1 \cap \partial V_2$ are both annuli the core of which are isotopic to $K$ and also, $V_2 \cap U = \varnothing$. We can define a $k$-function $f_1$ by modifying $f$ only within $\tilde{U} - U$ so that $V_1$ contains a $k$-model neighborhood of the source $K_1$ of $f_1$ and $V_2$ becomes a $k$-model neighborhood of the saddle $K_2$ of $f_1$. Here, $K$ is isotopic to the source $K_1$.

\begin{figure}[h]
\begin{center}
\includegraphics[scale = 0.5]{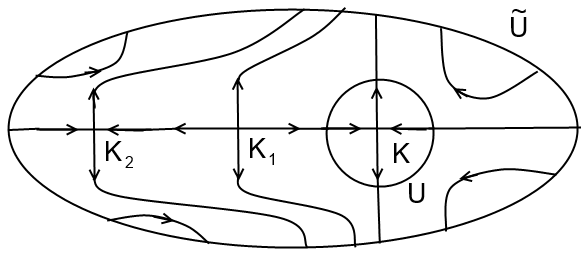} 
\caption{The figure shows some of the trajectories of a gradient like vector field for $f_1$.}
\label{source figure}
\end{center}
\end{figure}
\end{proof}

\begin{lem} \label{no saddle} Let $f$ be a $k$-function without any saddles. Then, $f$ has a single source and a single sink which form a Hopf link in $S^3$.

\end{lem}

\begin{proof} Since $S^3$ is closed, $f$ has at least one source $K_1$ and a sink $K_2$. Let $X$ be a gradient like vector field for $f$ and $X_t$ be the flow of $X$. Take a point $p$ in  $\partial U_1$ where $U_1$ is a $k$-model neighborhood of $K_1$. The point $X_t(p)$ will be in a $k$-model neighborhood $U_3$ of a sink $K_3$ of $f$ for large enough $t$ since $f$ has no saddles. Say, $X_{a_p}(p) \in U_3$ for some $a_p > 0$. Since $\partial U_1$ is compact and connected and $f$ does not have any saddles, we have $X_a(\partial U_1) \subseteq U_3$ for some time $a \geq a_p$. We may assume that $X_a(\partial U_1) = \partial U_3$ after scaling $X$ with a positive smooth function on $S^3$ if necessary. Then, $X_{a}(U_1) \cup U_3$ is an embedded, closed and connected 3-manifold in the closed and connected $S^3$. Therefore, $X_{a}(U_1) \cup U_3$ is $S^3$. Hence, the source $K_1$ and the sink $K_3 = K_2$ are the only critical knots of $f$. Let $P_1$ and $P_3$ denote the core of the solid tori  of $X_{a}(U_1)$ and $U_3$ respectively. The union $X_{a}(U_1) \cup U_3$ gives a lens space description of $S^3$ so that the two solid tori $X_{a}(U_1)$ and $U_3$ are two complementary standard solid tori in $S^3$ by the topological classification of lens spaces. Therefore, $P_1 \cup P_3 \simeq K_1 \cup K_3$ is a Hopf link.
\end{proof}

The above lemma shows that an unknot is $k$-mate. The next two lemmas describe a way to construct other $k$-mate knots and as we will show later in Theorem \ref{k-mate}, all $k$-mate knots  arise in this way starting with the unknot.

A knot $K$ is a cable knot of $J$ if $K$ can be isotoped into $\partial U$ where $U$ is a closed tubular neighborhood of $J$ in $S^3$. Here, $K$ is allowed to bound a disk in $\partial U$ so that an unknot is a trivial cable knot of any knot $J$. Even when $K$ does not bound a disk in $\partial U$ so that $K$ is not a trivial cable knot of $J$, the cable knot $K$ can be a meridian of $J$ or a longitude of an unknot $J$ so that $K$ is still a trivial knot. We will use the notation $K \simeq J_{p,q}$ which says that the cable knot $K$ of $J$ is homologically $p$ longitudes plus $q$ meridians of $J$. We will sometimes conveniently suppress the coefficients $p$ and $q$ and use the notation $K \succ J$ instead.

\begin{lem} \label{cable} A cable knot $K$ of a $k$-mate knot $J$ is $k$-mate.
\end{lem}

\begin{proof}
If $K$ is trivial, then it is $k$-mate by Lemma 1. Otherwise, $K$ can be isotoped to be transverse to the meridian disks of $k$-model neighborhood $U$ of $J$. We may assume that $J$  is a source of a $k$-function $f$ by Lemma \ref{source}. The rest of the proof will follow exactly as in that lemma where $f$ gets modified only within $U$ but still preserving a smaller tubular neighborhood of $J$. The knot $K$ becomes another source and a saddle isotopic to $K$ gets inserted between $K$ and $J$. 
\end{proof}

A connected sum $K_1 \# K_2$ of two knots $K_1$ and $K_2$ is not well defined in general unless both $K_1$ and $K_2$ and their ambient spaces $S^3$'s are all oriented. One can regard $K_1$ and $K_2$ as a split link in the same ambient space $S^3$ and an orientation of this single $S^3$ can be fixed easily. However, a $k$-function does not induce a natural orientation on a critical knot of it. While we study $k$-functions, we will strictly work with \emph{unoriented knots}. The notation $K_1 \# K_2$ will then denote a knot in the set $\lbrace K_1^+ \# K_2^+ , K_1^+ \# K_2^- \rbrace$ of knots where $K_i^\pm$ specifies an orientation of $K_i$.

\begin{lem} \label{sum} A connected sum $K_1 \# K_2$ of $k$-mate knots $K_1$ and $K_2$ is $k$-mate.

\end{lem}

\begin{proof} The $k$-mate knots $K_1$ and $K_2$ are sources of some $k$-functions $f_1$ and $f_2$ by Lemma \ref{source} respectively. Let $S$ be a sphere in $S^3$ which yields the connected summands $K_1$ and $K_2$ of $K_1 \# K_2$. Let $\tilde{S}$ be a small closed tubular neighborhood of $S$ in $S^3$ so that $\tilde{S} \cap K_1 \# K_2$ is two unknotted arcs in $\tilde{S} \simeq S^2 \times [0,1]$. Let $V$ be a small closed tubular neighborhood of $K_1 \# K_2$ in $S^3$ such that $V \cap \partial \tilde{S}$ is four disjoint disks and also, $V \cup \tilde{S}$ is smoothly embedded in $S^3$. Let $J$ denote the core of the annulus $S - V$. The region $S^3- \text{Int}(V \cup \tilde{S})$ has two connected components each of which is diffeomorphic to the complement of $K_1$ or $K_2$ in $S^3$. Let $R_i$ denote the component of $S^3 - \text{Int}(V \cup \tilde{S})$ that is diffeomorphic to the complement of $K_i$. Let $U_i$ be a small closed $k$-model neighborhood of the source $K_i$ of $f_i$. We may assume that $S^3 - \text{Int}(U_i) = R_i$ by isotoping $U_i$ in $S^3$. See Figure \ref{sum figure}. By using the flow of the gradient like vector fields for $f_1$ and $f_2$, we may scale $f_1$ and $f_2$ and add some constants so that they agree in a small tubular neighborhood of $\partial U_1$ or $\partial U_2$ with  $f_1(\partial U_1) = f_2(\partial U_2) = 1$ (see e.g. \textbf{\cite{milnor1}} for such scaling of $f_i$).

We can now define a $k$-function $f$ such that: 

\begin{enumerate}[label=(\roman*)]

\item $V$ is a $k$-model neighborhood of the source $K_1 \# K_2$ of $f$ with $f(K_1 \# K_2) = 0$

\item $\tilde{S} - V$ contains a $k$-model neighborhood of the saddle $J$ of $f$ with $f(J) = 0.5$ and also, $S - V$ is a stable annulus of $J$ in $\tilde{S} - V$.

\item $f_{| U_i} = {f_i}_{| U_i}$

\end{enumerate}

Therefore, $K_1 \# K_2$ is $k$-mate.

\begin{figure}[h]
\begin{center}
\includegraphics[scale = 0.4]{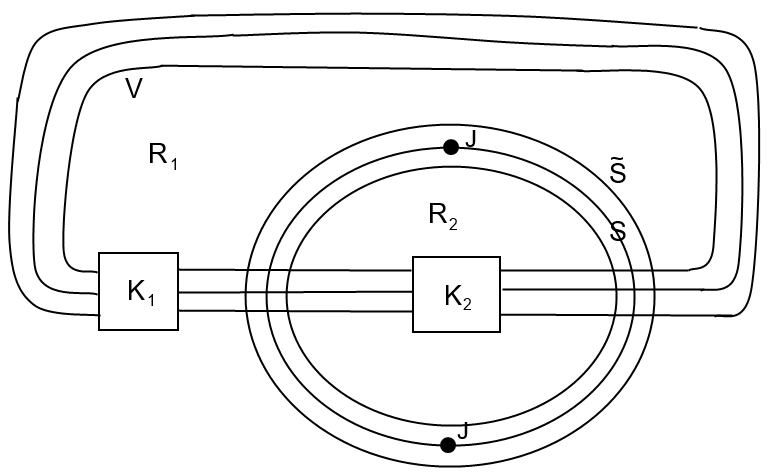}
\caption{} 
\label{sum figure}
\end{center}
\end{figure}
\end{proof}

\newpage

\section{Round handle decomposition of an ordered $k$-function}

\quad Before studying the preimage of a regular value of $f$ and how it changes when we pass a critical level, we first introduce \emph{ordered k-functions} where we make local modifications near the critical link $L$ of $f$ without changing the critical set of $f$ or the type of each component of $L$. The number $\epsilon > 0$ will denote a sufficiently small positive number throughout the text.

For a source $K$ of $f$ with local  $k$-model coordinates $(\theta, x, y)$, we can use an increasing smooth function $h: [0, 1] \rightarrow (-\infty, 0]$ with $h(z) = 0$ near $z=1$ and linear near $z=0$ to change $f$ locally by redefining $\tilde{f}(\theta,x,y) := f(\theta,x,y) + h(x^2+y^2) \leq f(\theta, x, y)$ so that $f$ can have arbitrarily small values on the source $K$. Similarly, $f$ can be redefined near a sink to have an arbitrarily large value on it. For a saddle $K$ of $f$, we can use a decreasing  (or increasing) smooth $h: [0, 1] \rightarrow [0, \epsilon]$ (or $[-\epsilon, 0]$) with small $\epsilon > 0$ and $h(z) = \pm \epsilon$ near $z=0$ and $h(z) = 0$ near $z = \pm \epsilon$. We can then redefine $f$ near $K$ as $\tilde{f}(\theta,x,y) = f(\theta,x,y) + h(x^2 + y^2)$ which changes the saddle value $f(K)$ by $\pm \epsilon$. Here, we have taken a small $\epsilon > 0$ to ensure that $\lvert h'(z) \lvert$ is also small and $K$ remains to be a saddle of $\tilde{f}$ without creating any other critical points.

\emph{An ordered $k$-function $f$} has then the following properties: 

\begin{enumerate}[label=(\roman*)]

\item The critical values of the critical knots of $f$ are all distinct. 

\item The critical values of $f$ are ordered as: $\text{source values} \leq \text{saddle values} \leq \text{sink values}$.

\end{enumerate}

Say, $a_1 < \dots < a_j < b_1 \dots < b_k < c_1  < \dots < c_m$ where $a_i, b_i$ and $c_i$  correspond to a source, a saddle and a sink of a ordered $k$-function $f$ respectively. Recall the smallness of $\epsilon$: if $z_0$ is a critical value of $f$, then $z_0$ is the only critical value of $f$ in $[z_0 - \epsilon, z_0 + \epsilon]$. We now describe a round handle decomposition \textbf{\cite{morgan}} of $S^3$  by analyzing the preimages of an ordered $k$-function $f$. Such an analysis will be repeatedly used in our proofs. 

Start with $r  < a_1$ having $f^{-1}((-\infty, r]) = \varnothing$. When we increase $r$, each time $r$ passes a source value of $f$, the preimage $f^{-1}([a_1, r])$ will have one more solid torus in $S^3$; \emph{a round 0-handle} is attached to the empty set. The region $f^{-1}([a_1, a_i + \epsilon])$ will consist of $i$ disjoint solid tori.

When we pass $b_1$, the preimage $\tilde{V} := f^{-1}([a_1, b_1 + \epsilon])$ is the union of $V:= f^{-1}([a_1, b_1 - \epsilon])$ which consists of $j$ disjoint solid tori and a region $f^{-1}([b_1 -\epsilon, b_1 + \epsilon])$. One connected component $R_1$ of this latter region is a solid torus that contains a $k$-model neighborhood of the saddle $K_1$ where $K_i$ is the saddle of $f$ with $f(K_i) = b_i$. The component $R_1$ contains a tubular neighborhood $\tilde{A}_1$ of the stable annulus of $K_1$ in $f^{-1}([b_1 -\epsilon, b_1 + \epsilon])$. For an appropriate choice of $\tilde{A}_1$, one can show that $f^{-1}([a_1, b_1 + \epsilon])$ is isotopic to $V \cup \tilde{A}_1$ in $S^3$ (see e.g. \textbf{\cite{milnor2}} for a Morse analogue of this fact). This tubular neighborhood $\tilde{A}_1$ is attached to $V$ along two disjoint annuli in $\partial V \cap \partial \tilde{A}_1$ the cores of which are isotopic to the unstable circles of $K$. The region $f^{-1}([a_1, b_1 + \epsilon])$ is topologically equivalent to the union of $f^{-1}([a_1, b_1 - \epsilon])$ and a solid torus $\tilde{A}_1$ in $S^3$ that intersect each other along two parallel annuli in $\partial \tilde{A}_1 \cap V$ the cores of which have $\pm 1$ longitude coefficients in $H_1(\partial \tilde{A}_1)$. In this case, $\tilde{A}_1$ is \emph{a round 1-handle} that is attached to $V$ along two annuli that are tubular neighborhoods of stable circles in $\partial V$. 

\begin{figure}[h]
\begin{center}
\includegraphics[scale = 0.8]{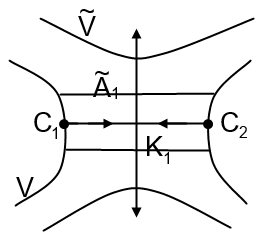}
\caption{} 
\label{saddle pass}
\end{center}
\end{figure}

The consecutive passes of saddle values of $f$ look similar. Each time we pass a saddle value $b_i$, the region $f^{-1}([a_1, b_i + \epsilon])$ is isotopic to the union $V: = f^{-1}([a_1, b_i - \epsilon])$  and a solid torus $\tilde{A}_i$  in $S^3$ where $V \cap \tilde{A}_i = \partial V \cap \partial \tilde{A}_i$ consists of two parallel annuli the cores of which have $\pm 1$ longitude coefficients in $H_1(\partial \tilde{A}_i)$. Here, the solid torus $\tilde{A}_i$ is a tubular neighborhood of a stable annulus of the saddle $K_i$. When we pass $b_i$, the boundary of the preimage changes from $\partial V$ to $\partial (V \cup \tilde{A}_i)$ by a surgery on the two stable circles  $C_1$ and $C_2$  of $K_i$ in $\partial V \cap \partial \tilde{A}_i$. Specifically, a tubular neighborhood $ \simeq S^1 \times \partial I \times I$ of $C_1$ and $C_2$ in $\partial V$ gets replaced by another two disjoint annuli $\simeq S^1 \times I \times \partial I$ which is now a tubular neighborhood of the unstable circles of $K_i$ in $\partial (V \cup \tilde{A}_1)$. They have the following identification where $I = [-1,1]$:

$S^1 \times \partial I \times (I - \lbrace 0 \rbrace) \overset{\simeq}{\longrightarrow} S^1 \times (I - \lbrace 0 \rbrace) \times  \partial I$

$(\theta, x , t y) \longrightarrow (\theta, t x , y)$ where $t \in (0,1]$ and $x,y = \pm 1$

We will use the notation $s( \cdot)$ to denote a surgered surface in $S^3$ coming from the pass of a saddle value of $f$ so that in the above situation, the surface $s(\partial V)$ is isotopic to $\partial (V \cup \tilde{A}_i)$ in $S^3$.

When we pass the first sink value $c_1$, a $k$-model neighborhood $U$ of the sink in $f^{-1}(c_1)$ fills in $V := f^{-1}([a_1, c_1 - \epsilon])$ in $S^3$. The region $f^{-1}([a_1, c_1 + \epsilon])$ is isotopic to $V \cup U$ in $S^3$ where $V \cap U = \partial V \cap \partial U$ is a torus. Here, $U$ is \emph{a round 2-handle} that is attached to $V$ along a torus.   So, the boundary of $f^{-1}([a_1, c_1 - \epsilon])$ is $m$ disjoint tori and each time we pass a sink value, one of those $m$ tori is filled in by a solid torus coming from a $k$-model neighborhood of a sink. This process ends with $f^{-1}([a_1, c_m]) = S^3$.

As we have explained above, an ordered $k$-function on $S^3$ induces a round handle decomposition of $S^3$. The converse is clear so that the correspondence between them is bijective. A knot $K$ is $k$-mate if and only if $K$ is the core of a round handle of some round handle decomposition of $S^3$ if and only if $K$ is a graph knot  \textbf{\cite{morgan}}. We will give an alternative proof of this fact but first a short remark.

Each ordered $k$-function naturally defines a Morse function on $S^3$ with the following indices of critical points: A pair of $0$ and $1$ from a source, a pair of $1$ and $2$ from a saddle and a pair of $2$ and $3$ from sink. For each such a $\lbrace j, j+1 \rbrace$ critical index pair, \emph{the attaching sphere} of the $j+1$-handle intersects \emph{the belt sphere} of the $j$-handle geometrically twice. 

The converse also holds. If $f_0: S^3 \rightarrow \mathbb{R}$ is a Morse function such that:

\begin{enumerate}[label=(\roman*)]

\item The critical values of $f_0$ are all distinct.

\item The indices of the critical points of $f_0$ come in adjacent pairs and their critical values are ordered on the real line $\mathbb{R}$ as: First $\lbrace 0,1 \rbrace$ pairs, then $\lbrace 1,2 \rbrace$ pairs and then $\lbrace 2,3 \rbrace$ pairs

\item For each $\lbrace j, j+1 \rbrace$ index pair of  critical  points $p_j$ and $p_{j+1}$ of $f_0$ respectively, the attaching sphere of the $j+1$-handle of $p_{j+1}$ intersects the belt sphere of the $j$-handle of $p_{j}$ geometrically twice. 

\end{enumerate}

Then, $f_0$ induces a ordered $k$-function $f$ on $S^3$. Each $\lbrace 0,1 \rbrace, \lbrace 1,2 \rbrace$ and $\lbrace 2,3 \rbrace$ index pair of paired critical points of $f_0$ gives rise to a source, saddle and a sink of $f$ respectively.

An alternative proof of the below lemma is in \textbf{\cite{morgan}}.

\begin{lem}  \label{torus} If $f$ is an ordered $k$-function and $r$ is a regular value of $f$, then each connected component of $f^{-1}(r)$ is an embedded torus in $S^3$.

\end{lem}

\begin{proof} Let $a_1 < \dots < a_j < b_1 \dots < b_k < c_1  < \dots < c_m$ denote the critical values of $f$ where $a_i, b_i$ and $c_i$  correspond to a source, a saddle and a sink of $f$ respectively. The lemma is clear for $r < b_1$ or $r > b_k$. Assume now that for some $b_1 < r < b_k$, the surface $f^{-1}(r)$ has a non-torus component. Since a torus has Euler characteristic 0 and the Euler characteristic of $f^{-1}(b_i-\epsilon)$ does not change after a surgery during the pass of $b_i$,  there exists a sphere $\hat{S}_1$ in some $f^{-1}(w_1)$ where $w_1$ is a regular value of $f$.

The sphere $\hat{S}_1$ bounds a 3-ball on each side in $S^3$ and let $B_1$ denote the one of them such that $B_1 \cap f^{-1}(w_1-\epsilon) = \varnothing$. As $f^{-1}(z)$ is a union of tori for a regular value $z$ with $z > b_k$, there must be a surgery on a sphere $S_1$ isotopic to $\hat{S}_1$  in $B_1$ during the pass of a saddle value but it may happen that the surface $s(S_1)$  produced  by surgery contains a sphere in $B_1$. Take a regular value $w_2 \geq w_1$ large enough such that $f^{-1}(w_2)$ contains a sphere $S_2$ in $B_1$ but the produced surface $s(S_2)$ does not contain a sphere after the pass of a saddle value $\beta_2$. Moreover, we can find such $w_2$ and $S_2$ such that the 3-ball $B_2$ bounded by $S_2$ in $S^3$ with $B_2 \cap f^{-1}(w_2 - \epsilon) = \varnothing$ satisfies $B_2 \subseteq B_1$. 

Let $K_2$ be the saddle with $f(K_2) = \beta_2$. Let $A_2$ denote the stable annulus of $K_2$ and $\lbrace C_1, C_2 \rbrace = \partial A_2$ denote the stable circles of $K_2$ with $C_1 \subseteq S_2$. Then, $C_2$ is not in $S_2$ but in another component $\Sigma$ of $f^{-1}(\beta_2 - \epsilon)$ with genus greater than 1 because $s(S_2)$ does not contain a sphere. 

The existence of such $\Sigma$ with big genus implies that $B_2$ contains at least one source $K$ and $B_2 \cap f^{-1}(b_1-\epsilon) \neq \varnothing$. Moreover, a surgery must happen on (not necessarily distinct) tori $T_a$ and $T_b$ in $B_2$ containing the stable circles $C_a$ and $C_b$ of a saddle respectively, such that both $C_a$ and $C_b$  bound disjoint disks in $T_a$ and $T_b$ respectively. This surgery then produces also a sphere $S_3$. Say, $S_3 \subseteq f^{-1}(w_3)$. Let $B_3$ denote the 3-ball bounded by $S_3$ in $S^3$ such that $B_3 \cap f^{-1}(w_3-\epsilon) = \varnothing$. We can find such $S_3$ and $B_3$ such that $B_3 \subseteq B_2$ and $K \subseteq B_2 - B_3$. We can now apply our last argument to $S_3 = \partial B_3$ instead of $S_1 = \partial B_1$  to conclude that $B_3$ contains at least one source $J$ with $J \neq K$, $B_4 \subseteq B_3$ and $J \subseteq B_3 - B_4$ where $B_4$ is a 3-ball in $S^3$ and the sphere $\partial B_4$ is in the preimage of a regular value of $f$. Therefore, $B_1$ contains infinitely many sources of $f$ and we have reached the desired contradiction.
\end{proof}

\newpage

\section{Characterization of $k$-mate knots}

\quad Suppose that $K_1$ and $K_2$ are two unknots such that $K_1$ is a cable knot of $K_2$. Since both $K_1$ and $K_2$ are unknots, $K_2$ is then a cable knot of $K_1$ as well. If it is a trivial cabling or $K_2$ is a longitude of $K_1$, then $K_1 \cup K_2$ is a split link of two unknots. If $lk( K_1, K_2) = \pm 1$, then $K_1 \cup K_2$ is a Hopf link. The below lemma exposes such cabled two unknots $K_1$ and $K_2$ in generality but we will encounter many split links of two unknots or Hopf links in its proof. 

\begin{lem} \label{unknot pair} Suppose that $K$ is a critical unknot of an ordered $k$-function $f$ and $K_R$ is the unknot core of an unknotted solid torus $R$ in $S^3$ such that $K \cap R = \varnothing$ and also, $\partial R$ is in the preimage of a regular value of $f$. Then, the unknots $K$ and $K_R$ are cable knots of each other. \end{lem}

\begin{proof}  If $f$ has no saddles, then Lemma \ref{no saddle} shows that $K \cup K_R$ is a Hopf  or a split link of two unknots. Assume now that $f$ has a saddle and let $a_1 < \dots < a_j < b_1 < \dots < b_k < c_1  < \dots < c_m$ denote the critical values of $f$ where $a_i, b_i$ and $c_i$  correspond to a source, a saddle and a sink of $f$ respectively. Let $K_1$ be the saddle of $f$ with $f(K_1) = b_1$. We will induct on the number of saddles $k$ by analyzing the stable circles $C_1$ and $C_2$ of $K_1$ in $f^{-1}(b_1 - \epsilon)$ and also a stable annulus $A$ of $K_1$ in $f^{-1}([b_1 - \epsilon, b_1])$. We may assume $r \notin [b_1 - \epsilon, b_1 + \epsilon]$ where $\partial R \subseteq f^{-1}(r)$. The circles $C_{1}$ and $C_{2}$ are cable knots of (not necessarily distinct)  sources $P_1$ and $P_2$ respectively. Let $E_i$ denote the solid torus component of $f^{-1}([a_1, b_1-\epsilon])$ containing $P_i$. Let $E$ denote the component of $f^{-1}([a_1, b_1+\epsilon])$ which contains $P_1 \cup P_2 \cup K_1$. In each case below, we will define an ordered $k$-function $f_1$ with at most $k-1$ saddles.

\begin{case} Only one of $C_{1}$ and $C_{2}$ bounds a disk in $f^{-1}(b_1-\epsilon)$. \end{case}

Say, $C_{1}$ bounds a disk $D$ in $f^{-1}(b_1-\epsilon)$ so that $K_1 \simeq C_1$ is an unknot. Then, $P_2$ cannot be nontrivial. Otherwise,  $C_{2}$ must be a meridian of $P_2$ and $P_2$ will intersect the sphere, which is the union of $D$, $A$ and a meridian disk of $P_2$,  geometrically once. As $H_2(S^3) = 0$, such a single geometric intersection of a 1-cycle and a 2-cycle of $S^3$ is not possible. So, $P_2$ is an unknot and $C_{2}$ is a longitude of $P_2$ as it bounds the disk $D \cup A$ in the complement of $P_2$. 

We consider the situation $P_1 \neq P_2$ first. The region $E$ is then isotopic to $E_1$ in $S^3$. If $K$ is equal $K_1$ or $P_2$, then $K$ is contained in a small 3-ball $B$ containing the disk $A \cup D$ such that $B \cap K_R = \varnothing$. The link $K \cup K_R$ is then a split link of two unknots. 

Assume now that $K$ is distinct from $K_1$ and $P_2$. We define a $k$-function $f_1$ by $f_1(p) := f(p)$ for $p \notin E$ so that $E$ becomes a $k$-model neighborhood of the source $P_1$ of $f_1$. Then, $f_1$ is ordered because the region $f^{-1}((-\infty,b_1+\epsilon])$ contains just a single saddle of $f$. If $E \cap R = \varnothing$, then $r$ is a regular value of $f_1$ with $\partial R \subseteq f_1^{-1}(r)$. An application of the induction hypothesis to the critical unknot $K$ of $f_1$ and the unknotted solid torus $R$ proves the lemma.

If $R \subseteq E$, then $R$ is a tubular neighborhood of $P_2$ or an unknot $P_1$. If $R$ is a tubular neighborhood of $P_2$, then $K_R \simeq P_2$ is contained in a 3-ball not containing $K$ and $K \cup K_R$ is a split link of two unknots. If $R$ is a tubular neighborhood of an unknot $P_1$, then $K \neq P_1$. Also, $R$ is isotopic to $E$ in $S^3$ and $K \nsubseteq E$ as $K$ is distinct from $K_1$ and $P_2$. So, the induction hypothesis applies to the critical unknot $K$ of $f_1$ and the unknotted solid torus $E$ to prove the lemma in this situation.

We now prove the lemma for the situation $P_1 = P_2$. Then, $E$ is a solid torus and $P_1 \cup K_1$ is a split link of two unknots. If $K$ is equal to $P_1$ or $K_1$, then $K$ is inside a 3-ball $B$ in $E$ such that $B \cap K_R = \varnothing$ so that $K \cup K_R$ is a split link of two unknots. Assume now $K \nsubseteq E$.  Let $K_E$ denote the core of $E$. We define $f_1$ by $f_1(p) := f(p)$ for $p \notin E$ so that $E$ becomes a $k$-model neighborhood of the source $K_E$ of $f_1$.  If $E \cap R=\varnothing$, we can apply the induction hypothesis as before. If $R \subseteq E$, then $R$ is either a $k$-model neighborhood of $P_1$ or $K_E$ is an unknot and both $R$ and $E$ are tubular neighborhoods of $K_E$. In the former case,  $K_R$ is contained in a 3-ball within $E$ not containing $K$ so that $K \cup K_R$ is a split link of two unknots. In the latter case, we apply the induction hypothesis just as before where $K \cup K_E \simeq K \cup K_R$.

\begin{case} Both $C_{1}$ and $C_{2}$ bound disks $D_1$ and $D_2$ in $f^{-1}(b_1-\epsilon)$ respectively. \end{case}

Then, $K_1$ is an unknot saddle. The disks $D_1$ and $D_2$ cannot be disjoint since otherwise, $f^{-1}(b_1 + \epsilon)$ would contain a sphere contradicting Lemma \ref{torus}. Therefore, $P_1 = P_2$. Say, $D_2 \subseteq D_1$. Let $A_1$ denote the annulus $D_1 - \text{Int}(D_2)$. The torus $T_0 = A_1 \cup A$ separates $S^3$ into two closed regions and let $R_0$ denote the one of them such that $ \text{Int} (R_0) \cap E_1 = \varnothing$.  Similarly, let $\tilde{R}_0$ denote the the component of $S^3 - \text{Int}(E)$ such that $\tilde{R}_0$ is isotopic to $R_0$ in $S^3$.

As $R_0$ in $S^3$ is bounded by a torus, it is diffeomorphic to the complement of a knot $K_0$ in $S^3$ (after smoothing the corners of $T_0$). Take a small 3-ball identified with $D_2 \times [0,1]$ coming from the push off of the disk $D_2$ into the exterior of $E_1$ in a normal direction so that $D_2 \times \lbrace 0 \rbrace := D_2 \subseteq \partial E_1$ and $\partial D_2 \times [0,1] \subseteq A$.  We can first take $K_0$ to be the union of a properly embedded arc in $A_1$ and another properly embedded arc in $A$. We can then slightly push off this union of two arcs into the exterior of $R_0$ in a normal direction to achieve $K_0 \cap R_0 = \varnothing$. Then, $S^3 - R_0$ is a tubular neighborhood of $K_0$ and $D_2 \times \lbrace 0 \rbrace$ is a meridian disk of $K_0$. Therefore, $B_0 := R_0 \cup D_2 \times [0,1]$ is diffeomorphic to a 3-ball which intersects $E_1$ in the disk $D_1$. Hence, the region $E_1 \cup R_0 \cup D_2 \times [0,1]$ is isotopic to $E_1$ in $S^3$ and similarly, so is the region $E \cup \tilde{R}_0$. So, the surface $s(\partial E_1)$, which has two components, has one component isotopic to $\partial E_1$ and another component isotopic to $T_0$ in $S^3$.

If $K = K_1$, then $K$ can be isotoped along $A$ into the 3-ball  $D_2 \times [0,1]$ and we can easily isotope $K_R$ out of this 3-ball if necessary without removing $K$ from that 3-ball. So, $K \cup K_R$ is a split link of two unknots. We will assume $K \neq K_1$ from now on.

If $(K \cup R) \cap (E \cup \tilde{R}_0) = \varnothing$, we define $f_1$ by $f_1(p) := f(p)$ for $p\notin E \cup \tilde{R}_0$ so that $E \cup \tilde{R}_0$ becomes a $k$-model neighborhood of the source $P_1$ of $f_1$. An application of the induction hypothesis to the critical unknot $K$ of $f_1$ and the solid torus $R$ proves the lemma.

If $(K \cup R) \subseteq \tilde{R}_0$, we define $f_1$ by $f_1(p) := f(p)$ for $p \in \tilde{R}_0$ so that $S^3 - \tilde{R}_0$ becomes a $k$-model neighborhood of the source $K_0$ of $f_1$. We can then apply the induction hypothesis just as before. 

Assume now that only one of $R$ and $K$ is inside $\tilde{R}_0$ and the other is outside $\tilde{R}_0$. Say, $K_a \subseteq \tilde{R}_0$ where $\lbrace K_a, K_b \rbrace = \lbrace K_R , K \rbrace$. Then, $K_a$ is contained in the 3-ball $B_0$ but $K_b$ is not so that $K \cup K_R$ is a split link of two unknots.

The final possible situation is that only one of $R$ and $K$ is inside $E$ but none of them are inside $\tilde{R}_0$. Then, $P_1$ is either equal to $K$ or isotopic to $K_R$ within $R$ and in the latter case, we may assume $P_1 = K_R$. Say, $K_a = P_1$ where $\lbrace K_a, K_b \rbrace = \lbrace K_R , K \rbrace$. Then, $K_b \nsubseteq E \cup\tilde{R}_0$. We define $f_1$ by  $f_1(p) := f(p)$ for $p\notin E \cup\tilde{R}_0$ so that $E \cup\tilde{R}_0$ becomes a $k$-model neighborhood of the source $P_1$ of $f_1$. An application of the induction hypothesis proves the lemma.

\begin{case} None of  $C_{1}$ and $C_{2}$ bounds a disk in $f^{-1}(b_1-\epsilon)$. \end{case}

\begin{subcase} Both $C_{1}$ and $C_{2}$ bound meridian disks $D_1$ and $D_2$ of $P_1$ and $P_2$ respectively. \end{subcase}

Then, $K_1$ is an unknot saddle. The sources $P_1$ and $P_2$ are equal since otherwise, $P_1$ would intersect the sphere $S_1 := D_1 \cup D_2 \cup A$ geometrically once in $S^3$. The sphere $S_1$ yields $P_1 \simeq P_a \# P_b$ and the region $S^3 - E$ has two components $R_a$ and $R_b$ which are isotopic to the complement of $P_a$ and $P_b$ in $S^3$ respectively. 

We consider the situation $K = K_1$ first. If $R \subseteq E$, then $P_1$ is an unknot and $R$ is a tubular neighborhood of it. We see that $K \cup K_R$ is a Hopf link in this setting. If $R \nsubseteq  E$, then say $R \subseteq R_b$. Since $K$ is in the 3-ball $B_a$ bounded by $S_1$ and containing the region $R_a$ but $K_R$ is not in $B_a$, the link $K \cup K_R$ is a split link of two unknots. We will assume $K \neq K_1$ from now on.

Assume $R \subseteq E$ so that $R$ is a tubular neighborhood of the unknot $P_1$ and $P_a$ and $P_b$ are unknots as well. As $K \neq K_1$, the unknot $K$ is either in $R_a$ or $R_b$. Say, $K \subseteq R_a$. We define $f_1$ by $f_1(p) := f(p)$ for $p \in R_a$ so that $S^3 - R_a$ becomes a $k$-model neighborhood of the unknot source $P_a$ of $f_1$. We can now apply the induction hypothesis to the critical unknot $K$ of $f_1$ and a $k$-model neighborhood of the source $P_a$ of $f_1$ to conclude that $P_a$ and $K$ are cable knots of each other. The lemma then follows because $P_a$ can be isotoped to $K_R$ within $S^3 - R_a$ so that $P_a \cup K \simeq K_R \cup K$.

If $R \nsubseteq E$, then say $R \subseteq R_b$. If $K \subseteq R_b$, we define $f_1$ by $f_1(p) := f(p)$ for $p \in R_b$ so that $S^3 - R_b$ becomes a $k$-model neighborhood of the source $P_b$ of $f_1$. An application of the induction hypothesis proves the lemma. If $K \subseteq R_a$, then $K$ is in the 3-ball $B_a$ not containing $R$ so that $K \cup K_R$ is a split link of two unknots. If $K \nsubseteq R_a \cup R_b$, then $K= P_1$ because $K \neq K_1$ as well. This final situation is similar to the previous situation where $R \subseteq E$ and $K \subseteq R_a \cup R_b$.

\begin{subcase} Only $C_1$ bounds a meridian disk $D_1$ of $P_1$. \end{subcase}

Then, $K_1$ is an unknot saddle. The sources $P_1$ and $P_2$ are distinct since $C_2$ is a non-meridian, nontrivial cable knot of $P_2$. Since $C_2$ bounds the disk $D_1 \cup A_1$, both $C_2$ and $P_2$ are unknots and $C_2$ is a longitude of $P_2$. Also, $P_2$ is a meridian of $P_1$ and $E$ is isotopic to $E_1$ in $S^3$.

We first consider the situation where $K$ is equal to $K_1$ or $P_2$. If $R \subseteq E$, then $R$ is a tubular neighborhood of either $P_1$ or $P_2$. If $R\supseteq P_2$, then $K \cup K_R \simeq C_2 \cup P_2$ is a split link of two unknots. If $R \supseteq P_1$, then $P_1$ is an unknot and $K \cup K_R \simeq C_1 \cup P_1$ is a Hopf link. If $R \nsubseteq E$, then $K$ is contained in a small 3-ball $B$ inside $E$ with $K_R \cap B = \varnothing$ so that $K \cup K_R$ is a split link of two unknots. We will assume that $K$ is distinct from $K_1$ and $P_2$ from now on. 

Assume $R \nsubseteq E$. We define $f_1$ by $f_1(p) := f(p)$ for $p \notin E$ so that $E$ becomes a $k$-model neighborhood of the source $P_1$ of $f_1$. The induction hypothesis can then be applied as before. 

Assume now $R \subseteq E$ so that $R$ is a tubular neighborhood of either $P_1$ or $P_2$. If $K = P_1$, then $K_R$ is isotopic to $P_2 $ within $R$ and $K \cup K_R \simeq P_1 \cup  P_2$ is a Hopf link. Assume now $K \cap  E = \varnothing$. If $R \supseteq P_2$, then $K_R$ is isotopic to $P_2$ within $R$ where $P_2$ is in the 3-ball $B$ not containing $K$ so that $K \cup K_R \simeq K \cup P_2$ is a split link of two unknots. If $R \supseteq P_1$, we define $f_1$ by $f_1(p) := f(p)$ for $p \notin E$ so that $E$ becomes a $k$-model neighborhood of the unknot source $P_1$ of $f_1$. The induction hypothesis can be applied as before.

\begin{subcase} None of $C_1$ and $C_2$ is a meridian of $P_1$ and $P_2$ respectively. \end{subcase}

The isotopic cable knots $C_1$ and $C_2$ are, say, $C_1 \simeq (P_1)_{p,q}$ and $C_2 \simeq (P_2)_{r,s}$ where $p, r \neq 0$ as $C_i$ is not a meridian of $P_i$. We will first consider the situation $P_1 \neq P_2$. Let $\tilde{A}$ be a closed tubular neighborhood of the stable annulus $A$ so that the annulus $\tilde{C}_i := \tilde{A} \cap E_i$ becomes a tubular neighborhood of $C_i$ in $\partial E_i$ and also, $E$ is isotopic to $E_1 \cup E_2 \cup \tilde{A}$ in $S^3$. The boundary of $E_1 \cup E_2 \cup \tilde{A}$ is a torus which comes from the union of the two annuli $\partial E_i - \text{Int}(\tilde{C_i})$ ($i = 1,2$) and the two annuli that are parallel copies of $A$ in $\tilde{A}$. Let $R_1 := S^3 - \text{Int}(E_1 \cup E_2 \cup \tilde{A})$.

If $p = \pm 1$, then $P_1$ is isotopic to $C_1$ in $E_1$. As $C_1$ and $C_2$ are isotopic, we get $P_1 \succ P_2$ (i.e., $P_1$ is a cable knot of $P_2$).  Similarly, if $r = \pm 1$, then $P_2 \simeq K_1$ and $P_2 \succ P_1$. In these situations, $E$ is isotopic to $E_1$ (when $P_2 \succ P_1$) or $E_2$ (when $P_1 \succ P_2$) in $S^3$.

We will now prove that $P_1 \cup P_2$ is a Hopf link when $p,r \neq 0, \pm 1$. The solid torus $E_1$ admits a Seifert fibration with a single singular fiber $P_1$ of multiplicity $\lvert p \rvert$ so that the annulus $\tilde{C_1}$ becomes a union of regular fibers because $C_1$ is not a meridian of $P_1$. This fibration on $\tilde{C_1}$ extends to a regular Seifert fibration on $\tilde{A}$ because $C_1$ is a cable knot of $K_1$ with $C_1 \simeq (K_1)_{\pm 1, \beta}$.  We may assume that $\tilde{C}_2 \subseteq \partial \tilde{A}$ is a union of regular fibers since the cable knots $C_1$ and $C_2$ of $K_1$ have the same slope. This fibration on $\tilde{C}_2$ can then be extended to a Seifert fibration of $E_2$ with a single singular fiber $P_2$ of multiplicity $\lvert r \rvert$ because $C_2$ is not a meridian of $P_2$. So, the region $E_1 \cup E_2 \cup \tilde {A}$ becomes a Seifert fibered manifold over a disk with two singular fibers of multiplicities $\lvert p \rvert$ and $\lvert r \rvert$. The torus $\partial (E_1 \cup E_2 \cup \tilde {A})$ bounds $E_1 \cup E_2 \cup \tilde {A}$ and $R_1$ in $S^3$ at least one of which must be a solid torus. Since $\pi_1(E_1 \cup E_2 \cup \tilde {A}, *) \simeq <z, w ; z^p = w^r> \not\simeq \mathbb{Z}$, the region $R_1$ must be a solid torus. A regular fiber in $\partial (E_1 \cup E_2 \cup \tilde {A}) = \partial R_1$ is nontrivial there. It cannot bound a meridian disk in $R_1$ since otherwise $\pi_1(S^3,*) = \pi_1(E_1 \cup E_2 \cup \tilde {A} \cup R_1, *) \simeq <z, w ; z^p = w^r=1>\not\simeq 1$. Therefore, the Seifert fibration on $\partial R_1$ can be extended into $R_1$ with at most one singular fiber so that we obtain a Seifert fibration of $S^3 = E_1 \cup E_2 \cup \tilde {A} \cup R_1$ over a sphere with two or three singular fibers. It follows now from the classification of Seifert fibered manifolds that $S^3$ cannot have a Seifert fibration with three singular fibers over a sphere but only two so that $R_1$ must have a regular Seifert fibration (see e.g. \textbf{\cite{orlik}} or \textbf{\cite{hatcher}}). Moreover, if one takes the base sphere as the union of two disks each of which contains a point corresponding to a singular fiber $P_1$ or $P_2$, then those two disks will correspond to two complementary solid tori in $S^3$ so that the cores $P_1$ and $P_2$ of those two complementary solid tori form a Hopf link in $S^3$.

We first consider the cases $P_1 \succ P_2$ or $P_2 \succ P_1$ where $E$ is isotopic to $E_1$ or $E_2$ in $S^3$. Say, $P_a \succ P_b$ where $\lbrace P_a , P_b \rbrace = \lbrace P_1, P_2 \rbrace$. Assume $R \nsubseteq E$. We define $f_1$ by $f_1(p) := f(p)$ for $p \notin E$ so that $E$ becomes a $k$-model neighborhood of the source $P_b$ of $f_1$. If $K$ is distinct from $P_a$ and $K_1$, then an application of the induction hypothesis proves the lemma. If $K$ is equal to $P_a$ or $K_1$, then $P_b$ is also an unknot because $P_a \simeq K_1$ is a nontrivial, non-meridian cable knot of $P_b$. Moreover, $P_a$ is isotopic to $P_b$ within $E$. An application of the induction hypothesis to the critical unknot $P_b$ of $f_1$ and $R$ proves the lemma because $K \cup K_R \simeq P_b \cup K_R$.

Assume now $R \subseteq E$. Then, $R$ is a tubular neighborhood of (possibly both) $P_a$ or $P_b$. In either case, $P_b$ is an unknot because $P_a$ is a nontrivial, non-meridian cable knot of $P_b$. Also, $K_R$ is isotopic to $P_b$ within $E$. If $K \subseteq E$, then $K$ is equal to $P_a$, $P_b$ or $K_1$ and also, $K \cup K_R \simeq P_a \cup P_b$ which proves the lemma. If $K \nsubseteq E$, we define $f_1$ by $f_1(p):= f(p)$ for $p \notin E$ so that $E$ becomes a $k$-model neighborhood of the source $P_b$ of $f_1$. We apply the induction hypothesis just as before.

We consider the Hopf link $P_1 \cup P_2$ case now. We still have both $P_1 \succ P_2$ and $P_2 \succ P_1$ but $E$ is no longer isotopic to $E_1$ or $E_2$ in $S^3$. The torus knot $C_1 \simeq (P_1)_{p,q}$ is nontrivial since $p \neq 0, \pm 1$. The region $V := S^3 - \text{Int}(E)$ is a solid torus the core of which is isotopic to $C_1 \simeq K_1$. 

Assume $K \cup R \subseteq V$. Since the core of $V$ is nontrivial, each of the unknots $K$ and $K_R$ is contained in a 3-ball $B_K$ and $B_R$ inside $V$ respectively. Let $g: V \rightarrow S^3$ be an embedding such that $g(V)$ is a standard, unknotted solid torus in $S^3$. Then, both $g(K)$ and $g(K_R)$ are unknots since each of $K$ and $K_R$ is contained in a 3-ball inside $V$. We define $f_1$ by $f_1(p) := f(g^{-1}(p))$ for $p \in g(V)$ so that $S^3 - g(V)$ becomes a $k$-model neighborhood of an unknot source $K_g$ of $f_1$.  The unknots $g(K_R)$ and $g(K)$ are then cable knots of each other by the induction hypothesis. Therefore, so is the link $g^{-1}(g(K) \cup g(K_R)) = K \cup K_R$.

Assume now that only one of $K$ and $R$ is inside $V$. The one inside $V$ is then contained in a 3-ball not containing the other one so that $K \cup K_R$ is a split link of two unknots. The final remaining case is $K \cup R \subseteq E$ where $K \cup K_R \simeq P_1 \cup P_2$ is a Hopf link.

We will now prove this subcase of the lemma for the situation $P_1 = P_2$. Let $\tilde{A}$, $\tilde{C}_i$ and $C_1 \simeq (P_1)_{p,q}$ \ $(p \neq 0)$ be just as before where $C_1$ is now isotopic to $C_2$ in $\partial E_1$. The nontrivial circles $C_1$ and $C_2$ separates $\partial E_1$ into two closed annuli $A_1$ and $A_2$ and the components of $s(\partial E_1)$ are isotopic to the tori $\Sigma_1 := A_1 \cup A$ and $\Sigma_2 := A_2 \cup A$ in $S^3$. Let $H_i$ denote the closed region bounded by $\Sigma_i$ in $S^3$ such that $\text{Int}(H_i) \cap E_1 = \varnothing$. Similarly, let $\tilde{H}_i$ denote the component of $S^3 - \text{Int}(E)$ that is isotopic to $H_i$  in $S^3$. 

Assume that $H_1$ is not a solid torus. Then, $E_1 \cup H_2$ bounded by $\Sigma_1$ is a solid torus.  If $C_1$ bounds a disk in $E_1 \cup H_2$, then $C_1$ is the meridian of the core of $E_1 \cup H_2$ and also a longitude of the unknot  $P_1$ because $C_1$ is a nontrivial, non-meridian cable knot of $P_1$. However, $H_1$ would be an unknotted solid torus in this case. Hence, $C_1$ does not bound a disk in $E_1 \cup H_2$ so that  $E_1 \cup H_2$ admits a Seifert fibration with at most one single singular fiber where the annuli $A_1$ and $A$ in its boundary become a union of regular fibers. As $\partial A_2 = \partial A_1$ consists of two regular fibers, the annulus $A_2$ can then be isotoped into $\partial (E_1 \cup H_2)$ relative to its boundary in the Seifert fibered solid torus $E_1 \cup H_2$ so that $H_2$ is also a solid torus isotopic to $E_1 \cup H_2$. 

Therefore, at least one of $H_1$ and $H_2$, say $H_1$, is a solid torus. Let $K_H$ denote the core of both $H_1$ and $\tilde{H}_1$. The union $E_1 \cup H_1$ of two solid tori intersecting each other at an annulus in their boundaries is then similar to the union $E_1 \cup E_2 \cup \tilde{A}$ in our previous situation $P_1 \neq P_2$. Therefore, either $P_1 \cup K_H$ is a Hopf link with $C_1$ being a nontrivial torus knot or one of  the knots $P_1$ and $K_H$ is a cable knot of the other one. In the latter case, $E_1 \cup H_1$ is isotopic to $E_1$ or $H_1$ in $S^3$. In the former Hopf  link case, the region $H_2$ is also a solid torus the core of which is isotopic to $C_1$. 

We have $S^3 = E \cup \tilde{H}_1 \cup \tilde{H}_2$ where the interiors of those three regions are disjoint. There are various possibilities about where $K$ and $R$ might be. We start with the  assumption $R \cup K \subseteq \tilde{H}_2$. The case where $P_1 \cup K_H$ is a Hopf link and $C_1$ is a nontrivial torus knot has already been analyzed in the ``Hopf link $P_1 \cup P_2$"  situation before and the lemma holds in this case. If $P_1 \succ K_H$ or $K_H \succ P_1$ and also, $E_1 \cup H_1$ is isotopic to $E_1$ or $H_1$ in $S^3$, we define  $f_1$ by $f_1(p) := f(p)$ for $p \notin E \cup \tilde{H}_1$ so that $E \cup \tilde{H}_1$ becomes a $k$-model neighborhood of the source $P_1$ or $K_H$ of $f_1$. An application of the induction hypothesis proves the lemma.

Assume now $R \cup K \subseteq \tilde{H}_1$. The case where $K_H$ is nontrivial has already been analyzed in the ``Hopf link $P_1 \cup P_2$" situation before and the lemma holds in this case. If $K_H$ is an unknot, we apply the induction hypothesis more directly by simply defining $f_1$ by $f_1(p) := f(p)$ for $p \in \tilde{H}_1$ so that $S^3 - \tilde{H}_1$ becomes a $k$-model neighborhood of an unknot source of $f_1$.

Assume now that only one of $R$ and $K$ is in $\tilde{H}_1$ and the other one is in $\tilde{H}_2$. Say, $K_a \subseteq \tilde{H}_1$ where $\lbrace K_a, K_b \rbrace = \lbrace K_R, K \rbrace$. If $K_H$ is nontrivial so that $\tilde{H}_1$ is a knotted solid torus, then $K \cup K_R$ is a split link of two unknots. If $K_H \cup P_1$ is a Hopf link and $C_1$ is a nontrivial torus knot, then $\tilde{H}_2$ is a knotted solid torus and $K \cup K_R$ is again a split link of two unknots. The remaining situation is that $K_H$ is trivial and $E \cup \tilde{H}_1$ is isotopic to $\tilde{H}_1$ in $S^3$. Also, $\tilde{H}_2$ is an unknotted solid torus in $S^3$. Here, each of the unknots $K$ and $K_R$ are in one of the two complementary unknotted solid tori $E \cup \tilde{H}_1$ and $\tilde{H}_2$ but not in the same one. Let $J_H$ denote the unknot core of $\tilde{H}_2$ so that $K_H \cup J_H$ is a Hopf link. We define $f_1$ by $f_1(p) := f(p)$ for $p \in \tilde{H}_1$ so that $E \cup \tilde{H}_2$ becomes a $k$-model neighborhood of the source $J_H$ of $f_1$.  An application of the induction hypothesis shows that $K_a$ is a cable knot of $J_H$. Since $K_H \cup J_H$ is a Hopf link, $K_a$ is a cable knot of $K_H$ as well. A similar argument shows that the other unknot $K_b$  is a cable knot of $K_H$ as well. Since $K_a \subseteq \tilde{H}_1$ but $K_b \nsubseteq \tilde{H}_1$, the unknot $K_a$ can be isotoped into an arbitrarily small tubular neighborhood of $K_H$ without affecting $K_b$. Therefore, the unknots in $K_a \cup K_b = K_R \cup K$ are cable knots of each other as well. 

The only possibility we haven't considered so far is that $R$ or $K$ is inside $E$. If $K \cup R \subseteq E$, then $R$ is a tubular neighborhood of $P_1$ and $K = K_1$. The lemma follows from $K \cup K_R \simeq C_1 \cup P_1$ and $C_1 \succ P_1$. Assume now that only one of $K$ and $R$ is inside $E$.  If $R \subseteq E$, we may assume $K_R = P_1$. Say, $K_a \subseteq E$ where $\lbrace K_a , K_b \rbrace = \lbrace K, K_R \rbrace$. Then $K_a$ is equal to $P_1$ or $K_1$ and in the latter case, $P_1$ is also an unknot because the unknot $K_1$ is a nontrivial, non-meridian cable knot of $P_1$. 

First assume $K_b \subseteq \tilde{H}_2$. If $P_1 \cup K_H$ is a Hopf link and $C_1$ is a nontrivial torus knot, then $\tilde{H}_2$ is a knotted solid torus and $K \cup K_R$ is a split link of two unknots. Otherwise, $E \cup \tilde{H}_1$ is isotopic to $E_1$  in $S^3$ and we define $f_1$ by $f_1(p) := f(p)$ for $p \in \tilde{H}_2$ so that  $E \cup \tilde{H}_1$ becomes a $k$-model neighborhood of the unknot source $P_1$ of $f_1$. The unknots in $P_1 \cup K_b$ are then cable knots of each other by the induction hypothesis. So are the unknots in $K_1 \cup K_b$ because $K_1$ can be isotoped into an arbitrarily small neighborhood of $P_1$ without affecting $K_b$. Since $K_R \cup K$ is either equal to $P_1 \cup K_b$ or $K_1 \cup K_b$, the lemma is proven in this situation.

Assume now $K_b \subseteq \tilde{H}_1$. If $K_H$ is nontrivial, then $K \cup K_R$ is a split link of two unknots. If $P_1 \cup K_H$ is a Hopf link, the unknots $K$ and $K_R$ are in two complementary unknotted solid tori but not in the same one and the lemma has been proven in this situation above. The remaining case is that $K_H$ is trivial and $E \cup \tilde{H}_1$ is isotopic to both $E_1$ and $\tilde{H}_1$ in $S^3$. In this case, $\tilde{H}_2$ is a standard solid torus and $P_1 \cup J_H \simeq K_H \cup J_H$ is a Hopf link where $J_H$ is the core of  $\tilde{H}_2$.  We define $f_1$ by $f_1(p) := f(p)$ for $p \in E \cup \tilde{H}_1$ so that  $\tilde{H}_2$ becomes a $k$-model neighborhood of the unknot source $J_H$ of $f_1$. An application of the induction hypothesis shows $K_b \succ J_H$ so that $K_b$  can be isotoped into $\partial \tilde{H}_1$ within $ \tilde{H}_1$. Therefore, $K_b \succ P_1$ because $P_1 \cup J_H$ is a Hopf link. We now see $K_b \succ K_a$ as well even when $K_a \neq P_1$ since in this case $K_a = K_1$ and, $K_b = K_R$ can be isotoped into an arbitrarily small tubular neighborhood of $P_1$ without affecting $K_1$. This completes the proof of the lemma.
\end{proof}

We will take the following elementary description as a definition of a graph knot \textbf{\cite{gordon}}. 

\begin{defn} \label{graph knot} Let $\mathbb{S}_0 := \lbrace \text{unknot} \rbrace$. To define $\mathbb{S}_n$ inductively ($n \in \mathbb{N}$), assume that  $\mathbb{S}_k$ is defined for $0 \leq k < n$ and let $\mathbb{S}_n$ denote the set of cable knots of $P$ where $P$ is a connected sum of knots in $\mathbb{S}_{n-1}$. A knot in $\mathbb{S} := \cup_{k \in \mathbb{N}} \ \mathbb{S}_k$ is called \emph{a graph knot}.
\end{defn}

We have the following facts about the set of graph knots $\mathbb{S}$: The set $\mathbb{S}_1$ is the set of (trivial or nontrivial) torus knots. Since the cable knot $(K)_{1,r}$ of any knot $K$ is isotopic to $K$, we have $\mathbb{S}_{n+1} \supseteq \mathbb{S}_{n}$ for all $n \in \mathbb{N}$. If $\lbrace K_1 , \dots , K_m \rbrace \subseteq \mathbb{S}_n$, then $K_1 \# \cdots \# K_m \simeq (K_1 \# \cdots \# K_m)_{1,r} \in \mathbb{S}_{n+1}$. If we have a sequence of cable knots $U \prec K_1 \prec K_2 \prec \cdots \prec K_m$ where $U$ is an unknot, then $K_i \in \mathbb{S}_i$ for $1 \leq i \leq m$.

\begin{defn} Suppose that $K$ is a graph knot in $\mathbb{S}_n$. If $n = 0$, we define \emph{the graph knot kit} or shortly \emph{the graph kit of $K$}  to be the empty set. For $n>0$,  fix a (not necessarily unique) expression of $K$ with $K \simeq (P)_{q,r}$ and $P \simeq P_{K,1} \# \cdots \# P_{K,m}$ where $P_{K,i} \in \mathbb{S}_{n-1}$ for $1 \leq i \leq m$. We define $\Gamma(K)$ corresponding to this fixed expression of $K$ by $\Gamma(K) := \lbrace P_{K,1} , \dots , P_{K,m} \rbrace$. Let $\Phi_1 := \Gamma(K)$. For $1 \leq j < n$, we define $\Phi_{j+1}$ inductively by $\Phi_{j+1} := \Phi_j \cup \lbrace J : \ J \in \Gamma(H) \text{\ where \ } H \in \Phi_j \rbrace$ where the elements  $P_{K,i} \in \Gamma(K)$ and $P_{J,s} \in \Gamma(J)$ are distinct elements of $\Phi_{j+1}$ whenever $K$ and $J$ are distinct in $\Phi_j$. Then, we say that $\Phi_n$ is \emph{a graph kit of $K$} and also, $K$ is \emph{woven from the graph knots in $\Phi_n$}. 
\end{defn} 

Note that even when the expression $K \simeq (P)_{q,r}$ and $P \simeq P_{K,1} \# \cdots \# P_{K,m}$ of $K \in \mathbb{S}_n$ is unique, we can consider $K$ in $\mathbb{S}_{n+1} \supseteq \mathbb{S}_n$ and may produce a different graph kit of $K$. Also, distinct graph knots in a graph kit can be isotopic. For example, for the torus knots $T_{2,3}$ and $T_{2,5}$ in $\mathbb{S}_1$, the set $\Phi := \lbrace T_{2,3} , T_{2,5}, \text{unknot}_{2,3}, \text{unknot}_{2,5} \rbrace$ is a graph kit of $T_{2,3} \# T_{2,5}$ (where $\Phi$ is valid for any orientations of $T_{2,3}$ and $T_{2,5}$ defining  $T_{2,3} \# T_{2,5}$). 

If $\Phi$ is a finite collection of unoriented knots, then $\underset{J \in \Phi}{\#} J$ denotes a connected sum of the knots in $\Phi$ which are assigned an arbitrary orientation before their connected sum is taken. If $\Phi$ is a graph kit of some graph knot and $P$ is a nontrivial graph knot in $\Phi$, then $\Gamma_\Phi(P)$ will denote the finite set of graph knots such that $P$ is a cable knot of a connected sum of the graph knots in $\Gamma_\Phi(P)$ and also, $\Gamma_\Phi(P) \subseteq \Phi$. In this case, the orientations of the knots in the connected sum $\underset{J \in \Gamma_\Phi(P)}{\#} J$ are not arbitrary but in such a way  so that the graph knot $P$ becomes a cable knot of the graph knot $\underset{J \in \Gamma_\Phi(P)}{\#} J$.

\begin{lem} \label{thread} A graph knot $K$ is $k$-mate. \end{lem}

\begin{proof} Say, $K \in \mathbb{S}_n$. If $n = 0$, then $K$ is an unknot which is $k$-mate by Lemma \ref{no saddle}. Assume now that $K$ is nontrivial and $n > 0$. We will induct on $n$. Each graph knot in $\Gamma(K)$ is $k$-mate by the induction hypothesis. Applications of Lemma \ref{cable} and Lemma \ref{sum} to the graph knots in $\Gamma(K)$ show that $K$ is $k$-mate. 
\end{proof}

\begin{thm} \label{thread kit}

Suppose that $f$ is an ordered $k$-function. Then, every critical knot $K$ of $f$ is a graph knot. Moreover, there exist a graph kit $\Phi$ of $K$ such that:

\begin{enumerate}[label=(\roman*)]

\item Each graph knot $P$ in $\Phi$ is isotopic to the core of a solid torus $R_P$ where $\partial R_P \subseteq f^{-1}(r)$ for some regular value $r$ of $f$.

\item For each nontrivial graph knot $P$ in $\Phi$, there exists a solid torus $R_{\Gamma(P)}$  such that the core of $R_{\Gamma(P)}$ is isotopic to $\underset{J \in \Gamma_\Phi(P)}{\#} J$  and also, $\partial R_{\Gamma(P)}\subseteq f^{-1}(r)$ for some regular value $r$ of $f$. Moreover, the core of $R_P$ can be isotoped into $\partial R_{\Gamma(P)}$ in $S^3$.

\item There exists a solid torus $R_{\Gamma(K)}$ such that the core of $R_{\Gamma(K)}$ is isotopic to  $\underset{J \in \Gamma_\Phi(K)}{\#} J$ and also, $\partial R_{\Gamma(K)}\subseteq f^{-1}(r)$ for some regular value $r$ of $f$. Moreover, $K$ can be isotoped into $\partial R_{\Gamma(K)}$.

\end{enumerate}

\end{thm}

\begin{rem} Part \emph{(i)} of Theorem \ref{thread kit} does not say that a saddle $K$ of an ordered $k$-function $f$ is the core of a solid torus $R$ where $\partial R \subseteq f^{-1}(r)$ for some regular value $r$ of $f$ but only that there exists a graph kit $\phi$ of $K$ such that this is true for every graph knot in $\phi$. However, $K$ is not in $\phi$. 

\end{rem}

\begin{proof} There is nothing to prove if $K$ is trivial because we can take the empty set as a graph kit of $K$. If $f$ has no saddles, then $f$ has a single unknot source and a single unknot sink which form a Hopf link by Lemma \ref{no saddle} and the theorem holds in this case. Assume now that $K$ is nontrivial so that $f$ has saddles. Let $a_1 < \dots < a_j < b_1 < \dots < b_k < c_1  < \dots < c_m$ denote the critical values of $f$ where $a_i, b_i$ and $c_i$  correspond to a source, a saddle and a sink of $f$ respectively. We will apply the proof technique in Lemma \ref{unknot pair} to induct on the number $k$ of saddles of $f$. As we will cover the similar cases or subcases, we will omit some details which can be found in that proof.

Let $K_1$ be the saddle of $f$ with $f(K_1) = b_1$. Let $A$ be the stable annulus of $K_1$ in $f^{-1}([b_1-\epsilon,b_1])$ and $C_1$ and $C_2$ be the stable circles of $K_1$ in $f^{-1}(b_1 - \epsilon)$. The circles $C_{1}$ and $C_{2}$ are cable knots of (not necessarily distinct)  sources $P_1$ and $P_2$ respectively.  Let $E_i$ denote the solid torus component of $f^{-1}([a_1, b_1-\epsilon])$ containing $P_i$. Let $E$ denote the component of $f^{-1}([a_1, b_1+\epsilon])$ which contains $P_1 \cup P_2 \cup K_1$.

\begin{case} Only one of $C_{1}$ and $C_{2}$ bounds a disk in $f^{-1}(b_1-\epsilon)$. \end{case}

Say, $C_{1}$ bounds a disk $D$ in $f^{-1}(b_1-\epsilon)$ Then, $K_1$ and $P_2$ are unknots so that $K$ is distinct from $K_1$ and $P_2$ and also, $C_{2}$ is a longitude of $P_2$. 

If $P_1 \neq P_2$, then $E$ is isotopic to $E_1$ in $S^3$. We define $f_1$ by $f_1(p) := f(p)$ for $p \notin E$ so that $E$ becomes a $k$-model neighborhood of the source $P_1$ of $f_1$. We apply the induction hypothesis to the critical knot $K$ of $f_1$ to conclude that $K$ is a graph knot and also, there exists a graph kit $\Phi$ of $K$ satisfying the properties stated in the theorem  for $f_1$. We can then take this graph kit $\Phi$ of the critical knot $K$ of $f$ to prove the theorem.

If $P_1 = P_2$, then $E$ is a solid torus not containing $K$. Let $K_E$ denote the core of $E$. We define $f_1$ by $f_1(p) := f(p)$ for $p \notin E$ so that $E$ becomes a $k$-model neighborhood of the source $K_E$ of $f_1$. The rest of the proof continues just as in the previous $P_1 \neq P_2$ situation.

\begin{case} Both $C_{1}$ and $C_{2}$ bound disks $D_1$ and $D_2$ in $f^{-1}(b_1-\epsilon)$ respectively. \end{case}

Then, $K_1$ is an unknot so that $K \neq K_1$. The disks $D_1$ and $D_2$ are not disjoint and say, $D_2 \subseteq D_1$. Let $A_1$ denote the annulus $D_1 - \text{Int}(D_2)$. The torus $T_0 = A_1 \cup A$ separates $S^3$ into two closed regions and let $R_0$ denote the one of them such that $ \text{Int} (R_0) \cap E_1 = \varnothing$. Similarly, let $\tilde{R}_0$ denote the component of $S^3 - \text{Int}(E)$ such that $\tilde{R}_0$ is isotopic to $R_0$ in $S^3$. The region $R_0$ is diffeomorphic to the complement of a knot $K_0$ in $T_0$. Since $E_1 \cup R_0$ is isotopic to $E_1$ in $S^3$, we can define an ordered $k$-function $\tilde{f}$ by modifying $f$ within a small neighborhood $U$ of $R_0$ such that $\tilde{f}(p) = f(p)$ for $p \notin U$ and also, $\tilde{f}$ does not have any critical points in $U$. So, $\tilde{f}(p) = f(p)$ for $p \notin f^{-1}([b_1-\epsilon,b_1+\epsilon])$ and the saddle $K_1$ of $f$ is removed.

If $K \nsubseteq R_0$, we define $f_1$ by $f_1(p) := \tilde{f}(p)$. The induction hypothesis applies to the critical knot $K$ of $f_1$ so that $K$ is a graph knot and also, there exists a graph kit $\Phi$ of $K$ such that the solid tori corresponding to the graph knots in $\Phi$ satisfy the properties stated in the theorem for $f_1$. For the boundary of one those solid tori lying in $f_1^{-1}(r)$ for a regular value $r$ of $f_1$, we may assume that $r \notin [b_1-\epsilon,b_1+\epsilon]$. That boundary will then be inside $f^{-1}(r)$ as well so that the graph kit $\Phi$ of $K$ satisfies the properties stated in the theorem for $f$ as well.

If $K \subseteq R_0$, we define $f_1$ by $f_1(p) := f(p)$ for $p \in \tilde{R}_0$ so that $S^3 - \tilde{R}_0$ becomes a $k$-model neighborhood of the source $K_0$ of $f_1$. We apply the induction hypothesis just as before to conclude that $K$ is a graph knot and also, there exists a graph kit $\Phi$ of $K$ satisfying the properties stated in the theorem for $f_1$. That graph kit will satisfy those properties for $f$ as well except that there may be a tubular neighborhood of $K_0$ associated to a graph knot in $\Phi$. This tubular neighborhood may not exactly work for $f$ but it is isotopic to an appropriate tubular neighborhood $S^3 - \text{Int}(\tilde{R}_0)$ coming from $f$.

\begin{case} None of  $C_{1}$ and $C_{2}$ bounds a disk in $f^{-1}(b_1-\epsilon)$. \end{case}

\begin{subcase} Both $C_{1}$ and $C_{2}$ bound meridian disks $D_1$ and $D_2$ of $P_1$ and $P_2$ respectively. \end{subcase}

Then, $K_1$ is an unknot with $K \neq K_1$. We have $P_1 = P_2$ and the sphere $S_1 := D_1 \cup D_2 \cup A$ yields $P_1 \simeq P_a \# P_b$. The region $S^3 - E$ has two components $R_a$ and $R_b$ which are isotopic to the complement of $P_a$ and $P_b$ in $S^3$ respectively. 

Assume $K \neq P_1$. Say, $K \subseteq R_a$.  We define $f_1$ by $f_1(p) := f(p)$ for $p \in R_a$ so that $S^3 - R_a$ becomes a $k$-model neighborhood of the source $P_a$ of $f_1$. We then apply the induction hypothesis just as before.

Assume now $K = P_1$. We can use $f_1$ above and similarly define $f_2$ for the regions $R_b$ and $S^3 - R_b$ to conclude that $P_a$ and $P_b$ are graph knots and also, there exist graph kits $\Phi_a$ and $\Phi_b$ of $P_a$ and $P_b$ respectively such that $\Phi_a$ and $\Phi_b$ satisfy the properties stated in the theorem for $f_1$ and $f_2$ respectively. Since $P_1 \simeq P_a \# P_b$, the source $P_1$ is also a graph knot and $\Phi := \Phi_a \cup \Phi_b \cup \lbrace P_a, P_b \rbrace$ is a graph kit of $P_1$. 

For each graph knot $P$ in $\Phi_a \cup \Phi_b$, we already have a solid torus $R_P$ or $R_{\Gamma(P)}$ associated to it. The boundaries $\partial R_P$ or $\partial R_{\Gamma(P)}$ are then in $f^{-1}(r)$ for some regular value $r$ of $f$ except possibly when $R_P$ or $R_{\Gamma(P)}$ is a tubular neighborhood of $P_a$ or $P_b$ but this problem in those exceptional cases can be easily resolved by just picking a more appropriate tubular neighborhood $R_P$ or $R_{\Gamma(P)}$ of $P_a$ or $P_b$ in the beginning. For the graph knots $P_a$ and $P_b$ in $\Phi$, we associate the solid tori $E \cup R_b$ and $E \cup R_a$ to $R_{P_a}$ and $R_{P_b}$ respectively. Finally, we regard $K$ as the cable knot $(K)_{1,q}$ of itself and define $R_{\Gamma(K)} := E_1$. The collection of all these solid tori associated to the graph knots in the graph kit $\Phi$ of $K$ satisfies then the properties stated in the theorem for $f$. 

\begin{subcase} Only $C_1$ bounds a meridian disk $D_1$ of $P_1$. \end{subcase}

Then, $K_1$ is an unknot so that $K \neq K_1$. The sources $P_1$ and $P_2$ are distinct since $C_2$ is a non-meridian, nontrivial cable knot of $P_2$. Also, $P_2$ is an unknot and $C_2$ is a longitude of $P_2$. The region $E$ is isotopic to $E_1$ in $S^3$. This subcase is then similar to Case 1 and the proof in this case can be completed similarly. 

\begin{subcase} None of $C_1$ and $C_2$ is a meridian of $P_1$ and $P_2$ respectively. \end{subcase}

We will first consider the situation $P_1 \neq P_2$. The isotopic cable knots $C_1$ and $C_2$ are, say, $C_1 \simeq (P_1)_{p,q}$ and $C_2 \simeq (P_2)_{r,s}$ where $p, r \neq 0$ as $C_i$ is not a meridian of $P_i$.

If $p$ or $r$ is equal to $1$, then $P_1 \succ P_2$ or $P_2 \succ P_1$ and $E$ is a tubular neighborhood of $P_1$ or $P_2$. Say, $P_a \succ P_b$ where $\lbrace P_a , P_b \rbrace = \lbrace P_1, P_2 \rbrace$. Let $\lbrace E_a, E_b \rbrace := \lbrace E_1, E_2 \rbrace$ be such that $E_a$ and $E_b$ are tubular neighborhoods of $P_a$ and $P_b$ respectively. We define $f_1$ by $f_1(p) := f(p)$ for $p \in (S^3-E) \cup E_b$ so that $E$ becomes a $k$-model neighborhood of the source $P_b$ of $f_1$. If $K$ is distinct from $K_1$ and $P_a$, we can then apply the induction hypothesis to critical knot $K$ of $f_1$ to prove the theorem. If $K$ is equal to $P_a$ or $K_1$, then $K \succ P_b$. We apply the induction hypothesis to the source $P_b$ of $f_1$ to conclude that $P_b$ is a graph knot and there exists a graph kit $\Phi_b$ of $P_b$ satisfying the properties stated in the theorem for $f_1$. Since $K \succ P_b$, the knot $K$ is also a graph knot and also, $\Phi_b \cup \lbrace P_b \rbrace$ is a graph kit of $K$. The solid tori $R_P$ or $R_{\Gamma(P)}$ is already defined for a graph knot $P$ in $\Phi_b$. We define $R_{P_b} := E_b$ and $R_{\Gamma(K)} := E_b$. The collection of these solid tori satisfies then the properties stated in the theorem for $f$. 

If $p,r \neq 0,\pm 1$, then $P_1 \cup P_2$ is a Hopf link so that nontrivial $K$ is distinct from $P_1$ and $P_2$. If $K = K_1$, the saddle $K$ is a torus knot. The graph kit $\lbrace P_1 \rbrace$ of $K_1$ together with the solid tori $R_{P_1} := E_1$ and $R_{\Gamma(K)} := E_1$ proves the theorem. Assume now that $K$ is inside the solid torus $V: = S^3 - \text{Int}(E)$ where the core $K_V$ of $V$ is a nontrivial torus knot $\simeq (P_1)_{p,q}$. 

This is the situation where we will need the fact from Lemma \ref{unknot pair} and also, the utility of a graph kit satisfying the properties stated in the theorem as we now embed $V$ into $S^3$ by $g: V \rightarrow S^3$ such that $g(V)$ becomes a standard, unknotted solid torus in $S^3$. We define $f_1$ by $f_1(p) := f(g^{-1}(p))$ for $p \in g(V)$ so that $S^3 - g(V)$ becomes a $k$-model neighborhood of an unknot source $K_g$ of $f_1$. 

If $g(K)$ is trivial, then $K_g$ and $g(K)$ are cable knots of each other by Lemma \ref{unknot pair}. Therefore, $K$ is a nontrivial, non-meridian cable knot of $K_V$. The graph kit $\lbrace K_V, P_1 \rbrace$ of $K$ together with the solid tori $R_{K_V} := V, \ R_{\Gamma(K_V)} := E_1, \ R_{P_1} := E_1$ and $R_{\Gamma(K)} := V$ proves the theorem.

Assume now that $g(K)$ is nontrivial. An application of the induction hypothesis to the critical knot $g(K)$ of $f_1$ shows that $g(K)$ is a graph kit and also, it produces a graph kit $\Phi_g$ of $g(K)$ satisfying the properties stated in the theorem for $f_1$. For $P$ in $\Phi_g$, let $R_P$ and $R_{\Gamma(P)}$ (when $P$ is nontrivial) be the solid tori as stated in the theorem. Let $R_{\Gamma(g(K))}$ be the solid torus for $\Gamma_{\Phi_g}(g(K))$ as stated in the theorem. If $R_P$ is not a tubular neighborhood of $K_g$, we can assume  $R_P \subseteq g(V)$. If $R_P$ is a tubular neighborhood of $K_g$, then the standard solid torus $R_P$ can be replaced by the standard solid torus $g(V)$ because the unknotted solid tori $R_P$ and $g(V)$ are isotopic in $S^3$ and also, a cable knot of the unknot core of $R_P$ is a cable knot of the unknot core of $g(V)$ as well. Hence, we can assume that $R_P \subseteq g(V)$ for each $P \in \Phi_g$. Similarly, we can assume that $R_{\Gamma(P)} \subseteq g(V)$ for each nontrivial $P \in \Phi_g$ and also, $R_{\Gamma(g(K))} \subseteq g(V)$. 

Let $K_P$ denote the core of $R_P$ for $P \in \Phi_g$. Since the graph kit $\Phi_g$ is a collection of isotopy classes of knots, our desired knot  $g^{-1}(K_P)$ is \emph{not defined}. However, the knot $g^{-1}(K_P)$ is \emph{well defined} and will do the job. If $K_P$ is an unknot, then Lemma \ref{unknot pair} asserts that $K_P$ is a cable knot of $K_g$. Therefore, the knot $g^{-1}(K_P)$, which is the core of the solid torus $g^{-1}(R_P)$, is a cable knot of $K_V$ so that it is a graph knot. Let $\Phi_1 := \lbrace g^{-1}(K_P) : P \in \Phi_g \rbrace$. The properties of $\Phi_g$ stated in the theorem imply that every knot in $\Phi_1$ is a graph knot because $g^{-1}(K_P)$ is a graph knot for every unknot $P$ in $\Phi_g$.  

Let $\Phi_0 : = \lbrace P \in \Phi_g  : K_P \text{ is trivial but } g^{-1}(K_P) \text{ is not trivial}  \rbrace$. Let $\Phi_2 := \Phi_1 \cup \lbrace (K_V)_P : P \in \Phi_0 \rbrace \cup \lbrace (P_1)_P : P \in \Phi_0 \rbrace$, where the latter two sets contain just distinct copies of the same knots $K_V$ and $P_1$.  For each nontrivial graph knot $g^{-1}(K_P)$ in $\Phi_1$, we take $\Gamma(g^{-1}(K_P))$ as $\lbrace g^{-1}(K_J) : J \in \Gamma_{\Phi_g}(P) \rbrace$ (when $P$ is nontrivial) or $\lbrace (K_V)_P \rbrace$ (when $P$ is trivial). We also take $\Gamma((K_V)_P)$ as $\lbrace (P_1)_P \rbrace$. So, we have $\Gamma(J) \subseteq \Phi_2$ when $J$ is a nontrivial graph knot in $\Phi_2$. To each graph knot $g^{-1}(K_P)$ in $\Phi_1$, we associate the solid torus $g^{-1}(R_P)$. To each graph knot $(K_V)_P$ or $(P_1)_P$ in $\Phi_2$, we associate the solid tori $V$ or $E_1$ respectively. When $g^{-1}(K_P)$ in $\Phi_2$ is nontrivial, we take the solid torus $g^{-1}(R_{\Gamma(P)})$ (when $P$ is nontrivial) or $V$ (when $P$ is trivial) for $R_{\Gamma(g^{-1}(K_P))}$. Finally, we define $R_{\Gamma((K_V)_P)} := E_1$ for $(K_V)_P \in \Phi_2$ and $R_{\Gamma(K)} := g^{-1}(R_{\Gamma(g(K))})$. The collection of these solid tori associated to the trivial or nontrivial graph knots in $\Phi_2$ satisfies the properties stated in the theorem for the ordered $k$-function $f$.  These properties of the solid tori imply now that the knot $g^{-1}(g(K)) = K$ is a graph knot that is woven from the graph knots in $\Phi_2$.

We will now prove the theorem for the situation $P_1 = P_2$. The nontrivial circles $C_1$ and $C_2$ separates $\partial E_1$ into two closed annuli $A_1$ and $A_2$ and the components of $s(\partial E_1)$ are isotopic to the tori $\Sigma_1 := A_1 \cup A$ and $\Sigma_2 := A_2 \cup A$ in $S^3$. Let $H_i$ denote the closed region bounded by $\Sigma_i$ in $S^3$ such that $\text{Int}(H_i) \cap E_1 = \varnothing$. Similarly, let $\tilde{H}_i$ denote the component of $S^3 - \text{Int}(E)$ that is isotopic to $H_i$ in $S^3$. Then, at least one of $H_1$ and $H_2$, say $H_1$, is a solid torus. Let $K_H$ denote the core of $H_1$. The link $P_1 \cup K_H$ is either a Hopf link with $C_1$ being a nontrivial torus knot or one of $P_1$ and $K_H$ is a cable knot of the other one. In the latter cable knot cases, the region $H_1 \cup E_1$ is isotopic to $H_1$ or $E_1$ in $S^3$.

There are now several possibilities for the location of $K$ in regard of $\tilde{H}_1 \cup E \subseteq S^3$. We start with the assumption $K \nsubseteq \tilde{H}_1 \cup E$. If $P_1 \succ K_H$ or $K_H \succ P_1$, let $\lbrace K_a, K_b \rbrace := \lbrace P_1, K_H \rbrace$ be such that $K_a \succ K_b$. We define $f_1$ by $f_1(p) := f(p)$ for $p \notin \tilde{H}_1 \cup E$ so that $\tilde{H}_1 \cup E$ becomes a $k$-model neighborhood of the source $K_b$ of $f_1$. An application of the induction hypothesis proves the theorem. If $P_1 \cup K_H$ is a Hopf link, the region $S^3 - \tilde{H}_1 \cup E$ is a tubular neighborhood of a nontrivial torus knot isotopic to $K_1$ and we have already proven the theorem in this setting which was analyzed in the situation $P_1 \neq P_2$.

Assume now $K \subseteq E$ so that $K$ is equal to $P_1$ or $K_1$. We first consider $K = P_1$. As $K$ is nontrivial, $P_1 \cup K_H$ cannot be a Hopf link so that $H_1 \cup E_1$ is isotopic to $H_1$ or $E_1$ in $S^3$. Let $\lbrace K_a, K_b \rbrace$ and $f_1$  be defined just as in the previous paragraph. The induction hypothesis applies to the source $K_b$ of $f_1$ so that $K_b$ is a graph knot and also, there exists a graph kit $\Phi_b$ of $K_b$ satisfying the properties stated in the theorem  for $f_1$. If $K = K_b$, then this graph kit $\Phi_b$ works for $f$ as well. Otherwise, $K \succ K_b$ so that $K$ is a graph knot and $\Phi := \Phi_b \cup \lbrace K_b \rbrace$ is a graph kit of $K$. Set the solid tori $R_{K_b} := \tilde{H}_1$ and $R_{\Gamma(K)} := \tilde{H}_1$.  The collection of these two solid tori together with the solid tori associated to the graph knots in $\Phi_b$ proves the theorem.

Assume now $K = K_1$. Our previous argument shows that $P_1$ is a graph knot and when $P_1$ is not trivial, there exists a graph kit $\Phi_1$ of $P_1$ satisfying the properties stated in the theorem. If $P_1$ is trivial, then take $\Phi_1$ to be the empty set. Since $K \succ P_1$, the knot $K$ is a graph knot and $\Phi_1 \cup \lbrace P_1 \rbrace$ is a graph kit of $K$. The collection of the solid tori corresponding to the graph knots in $\Phi_1$ together with $R_{P_1} := E_1$ and $R_{\Gamma(K)} := E_1$ proves the theorem.

Assume now $K \subseteq \tilde{H}_1$. If $P_1 \succ K_H$ or $K_H \succ P_1$, we can define $\lbrace K_a, K_b \rbrace$ and $f_1$  just as before and our previous argument shows that $K_b$ is a graph knot. Also, we get a graph kit $\Phi_b$ of $K_b$ satisfying the properties stated in the theorem for $f_1$. Let $\lbrace U_a, U_b \rbrace := \lbrace E_1, \tilde{H}_1 \rbrace$ be such that $U_a$ and $U_b$ are tubular neighborhoods of $K_a$ and $K_b$ respectively. If $K_H = K_b$, we simply define $\Phi_H := \Phi_b$.  If $K_H = K_a$, we define a graph kit $\Phi_H := \Phi_b \cup \lbrace K_b \rbrace$ of $K_H$ and the solid tori $R_{K_b} := U_b$ and $R_{\Gamma(K_H)} := U_b$. The graph kit $\Phi_H$ of $K_H$ with its associated solid tori satisfies then the properties stated in the theorem. We now embed $\tilde{H}_1$ into $S^3$ by $g: \tilde{H}_1 \rightarrow S^3$ such that $g(\tilde{H}_1)$ becomes a standard, unknotted solid torus in $S^3$. Such an embedding $g$ onto a standard, unknotted solid torus has been studied before in the $P_1 \neq P_2$ situation. We can similarly prove the theorem in this situation by combining the graph kits $\Phi_H$ and $\Phi_g$ of $g(K)$. This completes the proof of the theorem.
\end{proof}

\begin{proof}[Proof of Theorem \ref{k-mate}] Lemma \ref{thread} proves one side of the theorem and Theorem \ref{thread kit} proves the other side since any $k$-function $f$ can be made ordered by modifying it within a tubular neighborhood of its critical link without changing the set of critical points of $f$.
\end{proof}

\end{document}